\documentclass[a4paper,12pt,reqno]{amsart}
\usepackage{mathrsfs,amsmath,amssymb,latexsym,amsfonts, amscd}
\title{On an efficient induction step with Nklt(X,D)--Notes to Todorov}
\author{Meng Chen}
\address{\rm Institute of Mathematics \& LMNS, Fudan University, Shanghai 200433, China}
\email{mchen@fudan.edu.cn}

\newcommand{\bQ}{{\mathbb Q}}
\newcommand{\bP}{{\mathbb P}}
\newcommand{\roundup}[1]{\lceil{#1}\rceil}
\newcommand{\rounddown}[1]{\lfloor{#1}\rfloor}

\newcommand\vol{\text{\rm vol}}
\newcommand\lrw{\longrightarrow}
\newcommand\rw{\rightarrow}
\newcommand\hrw{\hookrightarrow}

\newcommand\OO{{\mathcal{O}}}

\newcommand\eps{\varepsilon}

\newcommand{\tdr}{\tilde{r}}
\newtheorem{thm}{Theorem}[section]
\newtheorem{lem}[thm]{Lemma}
\newtheorem{cor}[thm]{Corollary}
\newtheorem{prop}[thm]{Proposition}

\theoremstyle{definition}

\theoremstyle{remark}

\begin{document}
%%%%%
\begin{abstract} Applying the effective induction on $\text{Nklt}(X,D)$ developed by Hacon--McKernan  and Takayama, Todorov proved that $\varphi_5$ is birational for projective 3-folds $V$ with $\vol(V)\gg 0$, which was recently improved by Di Biagio in loosing the volume constraint. The observation is that the least efficient induction step can be studied in an alternative way, which allows us to assert Todorov's statement for $\vol(V)>12^3$. The 4-dimensional analog is also given in this note. The idea works well for all dimensions.  
\end{abstract}
%%%%%%
\maketitle
%%%%%%%%%
\pagestyle{myheadings}
\markboth{\hfill M. Chen\hfill}{\hfill On an efficient induction step with $\text{Nklt}(X,D)$\hfill}
\numberwithin{equation}{section}
%%%%%%%%%%%%
%\tableofcontents
%%%%%%%%%

\section{\bf Introduction}

%The main results of this short note are Theorem \ref{m1}, Theorem \ref{m2} and Corollary \ref{m3}. 
We work over the complex number field ${\mathbb C}$.  One of the fundamental questions in birational geometry has been to find an optimal integer $r_n>0$ ($n\geq 3$) such that the pluricanonical map $\varphi_m$ is birational onto the image for all $m\geq r_n$ and for all $n$-dimensional nonsingular projective varieties of general type. It is well-known that $r_1=3$ and, due to Bombieri, $r_2=5$.  For 3-folds, one has $r_3\leq 73$ by Chen--Chen \cite{Ex1,JDG}. For all $n\geq 3$, the remarkable boundedness theorem, i.e. $r_n<+\infty$, was proved separately by Hacon--McKernan \cite{H-M}, Takayama \cite{Tak} and Tsuji.  One may refer to the very nice survey article \cite{HM10} for other boundedness results in birational geometry. 
\medskip

We refer to \cite{H-M, Tak, Joro} for those standard notions of ``volume'', ``lc center'' and ``Nklt(X,D)''.  In this paper we are mainly interested in the explicit boundedness of projective varieties with large volumes.

Utilizing the effective induction on non-klt centers (i.e. $\text{Nklt}(X,D)$) developed in Hacon--McKernan \cite{H-M} and Takayama \cite{Tak}, Todorov \cite[Theorem 1.2]{Joro} proved that $\varphi_5:=\Phi_{|5K_V|}$ is birational for all nonsingular projective 3-folds $V$ with $\vol(V)>4355^3$. Recently Di Biagio \cite[Theorem 1.2]{DB} improved this by loosing the volume constraint, say $\vol(V)>2\cdot 1917^3$. 

In fact, the least efficient induction step of Todorov is when $V$ is birationally fibred by a family of surfaces with small volumes over a curve. Recalling a relevant result in our paper \cite{Bonn} and using the ``canonical restriction inequality'' in Chen--Zuo \cite{CAG}, the mentioned induction step of Todorov can be handled quite well in an alternative way. This allows us to present the following improved statement:

\begin{thm}\label{m1} Let $V$ be a nonsingular projective 3-fold with $\vol(V)>12^3$. Then $\varphi_m$ is birational for all $m\geq 5$.
\end{thm}

It is worthwhile to realize the idea in higher dimensions. {}First of all we note that the MMP works well for projective varieties of general type due to, for instance, Birkar--Cascini--Hacon--McKernan \cite{BCHM}, Hacon--McKernan \cite{H-M-M} and Siu \cite{Siu}. When we speak of ``minimal models'' or ``minimal n-folds'', it should always be understood that the variety has $\bQ$-factorial terminal singularities.
\medskip

\noindent{\bf Definition}. For any integer $n\geq 2$, define the positive integer $\tilde{r}_n$ to be the minimum satisfying the following properties:
\begin{quote}{\em for all minimal projective n-folds $X_0$ of general type, for all birational morphisms $\nu: X'\rw X_0$ with $X'$ nonsingular and for all nef and big $\bQ$-divisors $R$ on $X'$, the linear system
$$|K_{X'}+\roundup{(m-2)\nu^*(K_{X_0})+R}|$$
gives a birational map onto the image for all $m\geq \tilde{r}_n$.}\end{quote}
\medskip

\noindent{\bf Remark}. We can see, later on, in Lemma \ref{1} and Lemma \ref{3+} that $\tdr_n<+\infty$ for all $n\geq 2$, which is essentially an induction on the dimension applying Kawamata--Viehweg vanishing and the boundedness theorem of Hacon--McKernan \cite{H-M}, Takayama \cite{Tak} and Tsuji. (In fact, it is not hard to see from Chen--Chen \cite{JDG} that $\tdr_3\leq 73$.)

%The existence of $\tilde{r}_n$ is due to Hacon--McKernan \cite{H-M} and Takayama \cite{Tak} by the effective induction.  When $n\leq 3$, one knows more:
%\begin{itemize}
%\item[(1)] $\tilde{r}_2=5$ by Lemma \ref{1};
%\item[(2)] $\tilde{r}_3\leq 73$ by virtue of Chen--Chen \cite{JDG}.
%\end{itemize}
\medskip

In the second part we shall prove the following:

\begin{thm}\label{m2} Let $Y$ be a nonsingular projective 4-fold. Assume $\vol(Y)\gg 0$, say $\vol(Y)>\frac{\tdr_3-1}{6}\cdot 12^4$. Then $|mK_Y|$ gives a birational map for all $m\geq \tdr_3$.
\end{thm}

A difficulty in high dimensions is that we do not know any practical lower bound for $\tdr_n$ ($n\geq 4$). We need to set the following:
\medskip

\noindent{\bf Notation.} For all $i>0$, the real number $e_i>0$ denotes the optimal constant so that $\vol(Z)>e_i^i$ for all $i$-folds $Z$ of general type. Set 
$$R_n:=\text{max}\{2\prod_{i=1}^{n-2}(\frac{i\sqrt[i]{2}}{e_i}+1)-1, \tdr_{n-1}\}.$$
\medskip

Our last result is the following:
\begin{cor}\label{m3} For all $n\geq 5$, $C_n$ are computable positive constants. Let $Y$ be a nonsingular projective n-fold with $\vol(Y)>C_n$. Then $|mK_Y|$ gives a birational map for all $m\geq R_n$. 
\end{cor}
\bigskip

Throughout we are in favor of the following symbols:
\begin{itemize}
\item``$\sim$'' denotes linear equivalence or ${\mathbb Q}$-linear equivalence; 
\item``$\equiv$'' denotes numerical equivalence; 
\item``$A\geq  B$'', for $\bQ$-divisors $A$ and $B$, means that $A-B$ is $\bQ$-linearly equivalent to an effective $\bQ$-divisor.
\item ``$|M_1|\succeq |M_2|$'' means, for linear systems $|M_1|$ and $|M_2|$,  
$|M_1|=|M_2|+(\text{fixed effective divisor}).$
\end{itemize}
\medskip

I would like to thank the anonymous referee for his(her) valuable suggestions and comments.

\section{\bf Proof of main statements}

We start with the following:

\begin{lem}\label{1} Let $S$ be a nonsingular projective surface of general type and $\sigma: S\rw S_0$ the birational contraction onto the minimal model $S_0$. Let $Q$ be any nef and big $\bQ$-divisor on $S$. Then the linear system $|K_S+3\sigma^*(K_{S_0})+\roundup{Q}|$ gives a birational map.
\end{lem}
\begin{proof} Write ${\mathcal L}:=3\sigma^*(K_{S_0})+Q$. Clearly ${\mathcal L}$ is nef and ${\mathcal L}^2>9$. 

Case 1. $S$ is not a (1,2) surface.  For any irreducible and reduced curve $\hat{C}\subset S$, passing through a very general point $P\in S$, it is sufficient to prove $(\sigma^*(K_{S_0})\cdot \hat{C})\geq 2$ and so that $({\mathcal L}\cdot \hat{C})>6$. Then the statement follows from either \cite[Lemma 2.6]{Bonn} or \cite[Proposition 4]{Masek}. In fact, $(\sigma^*(K_{S_0})\cdot \hat{C})=(K_{S_0}\cdot C)$ where $C:=\sigma_*(\hat{C})$.  We may verify this on $S_0$ assuming $C$ to be a curve in $S_0$ passing through a very general point of $S_0$. Note that $C$ is of general type. Thus $(K_{S_0}\cdot C)+C^2\geq 2$.  Suppose to the contrary that $(K_{S_0}\cdot C)\leq 1$. Then $C^2\geq 1$.  By the Hodge Index Theorem, one sees $K_{S_0}^2=C^2=1$ and $K_{S_0}\equiv C$.  The surface theory implies that $S$ is either a $(1,1)$ surface or a $(1,0)$ surface. If $(K_{S_0}^2, p_g(S))=(1,0)$, then the torsion element $\theta:=K_{S_0}-C$ is of order $\leq 5$ and $h^0(S_0, C)=1$. Thus there are at most finite number of such curves on $S_0$ since it is determined by the group $\text{Tor}(S)$ with $|\text{Tor}(S)|\leq 5$. By the choice of $C$, this is impossible. If $(K_{S_0}^2, p_g(S))=(1,1)$, then $K_{S_0}\sim C$ since $\text{Tor}(S_0)=0$ by Bombieri \cite[Theorem 15]{Bom} and thus $C$ is the unique canonical curve of $S_0$, which is impossible either by the choice of $C$. In a word, we have $(K_{S_0}\cdot C)\geq 2$. 

%If $S$ is not , then, for any curve $C$ passing through a very general point $P\in S$, one has $({\mathcal L}\cdot C)>6$ thanks to the Hodge index theorem for the case $K_{S_0}^2\geq 2$, Miyaoka \cite[Lemma 5]{Miy} for the case $(K_{S_0}^2, p_g(S))=(1,0)$ and Bombieri \cite{Bom} for the case $(K_{S_0}^2, p_g(S))=(1,1)$. 

Case 2. $S$ is a (1,2) surface.  The statement follows from \cite[Lemma 1.3, Lemma 1.4]{JPAA}.
\end{proof}

\begin{prop}\label{3} Let $Y$ be a nonsingular projective 3-fold of general type, $B$ a smooth projective curve and $h: Y\rw B$ be a fibration. Denote by $F$ a general fiber of $h$. Assume $K_Y\sim_{\bQ} pF+E_p$ for some rational number  $p>5$ and an effective $\bQ$-divisor $E_p$. Then $|mK_Y|$ gives a birational map for all $m\geq 5$. 
\end{prop}
\begin{proof} Take a minimal model $Y_0$ of $Y$. Modulo birational modifications over $Y$, we may assume for simplicity that there is a birational morphism $\pi:Y\rw Y_0$. Take a sufficiently large and divisible integer $m>0$ such that 
\begin{itemize}
\item[(i)] $mK_{Y_0}$ is a Cartier divisor;
\item[(ii)] both $|mK_{Y_0}|$ and $|pmF|$ are base point free;
\item[(iii)] $mK_Y\sim_{\mathbb Z}mpF+mE_p$;
\item[(iv)] the support of the union of $E_p$ and all those exceptional divisors of $\pi$ is simple normal crossing.
\end{itemize}
Then one has $pmF\leq m\pi^*(K_{Y_0})\leq mK_Y$. By Chen--Zuo \cite[Lemma 3.3, Lemma 3.4]{CAG}, giving any small rational number $\eps>0$, one has
\begin{equation}\label{ie1}\pi^*(K_{Y_0})|_F\equiv (\frac{p}{p+1}-\eps)\theta^*(K_{F_0})+H_{\eps}\end{equation}
for certain effective $\bQ$-divisor $H_{\eps}$ on $F$ where $\theta: F\rw F_0$ is the contraction onto the minimal model. Write $\pi^*(K_{Y_0})\sim_{\bQ}pF+E_p'$ where $E_p'$ is another effective $\bQ$-divisor.  By assumption the support of $E_p'$ is simple normal crossing.

Pick two distinct general fibers $F_1$, $F_2$ of $h$. Both $F_1$ and $F_2$ are known to be nonsingular projective surfaces of general type. Since 
$$4\pi^*(K_{Y_0})-F_1-F_2-\frac{2}{p}E_p'\equiv (4-\frac{2}{p})\pi^*(K_{Y_0})$$ is nef and big, Kawamata--Viehweg vanishing (\cite{KV,VV}) gives the surjective map
\begin{equation}\label{1.1}
\begin{array}{rl}
&H^0(K_Y+\roundup{K_{4,p}})\\
\rw&
H^0(F_1, K_{F_1}+\roundup{K_{4,p}-F_1}|_{F_1})
\oplus H^0(F_2, K_{F_2}+\roundup{K_{4,p}-F_2}|_{F_2})
\end{array}\end{equation}
where $K_{4,p}:=4\pi^*(K_{Y_0})-\frac{2}{p}E_p'$.
Observing that
$$\roundup{4\pi^*(K_{Y_0})-\frac{2}{p}E_p'-F_i}|_{F_i}\geq 
\roundup{L_{4,i}}$$
where $$L_{4,i}:=(4\pi^*(K_{Y_0})-\frac{2}{p}E_p'-F_i)|_{F_i}\equiv (4-\frac{2}{p})\pi^*(K_{Y_0})|_{F_i}$$ for $i=1,2$,  we have
$$|5K_Y||_{F_i}\succeq |K_Y+\roundup{4\pi^*(K_{Y_0})-\frac{2}{p}E_p'}||_{F_i}\succeq |K_{F_i}+\roundup{L_{4,i}}|.$$
Due to the surjectivity of (\ref{1.1}), it is sufficient to show that $|K_{F_i}+\roundup{L_{4,i}}|$ gives a birational map for $i=1,2$. 

Note that $\frac{p}{p+1}>\frac{5}{6}$ whenever $p>5$. 
Take a small rational number $\eps_0>0$ such that $\delta_0:=(4-\frac{2}{p})(\frac{p}{p+1}-\eps_0)-3>0$. Then we have
$$L_{4,i}\equiv (3+\delta_0)\theta_i^*(K_{F_{i,0}})+(4-\frac{2}{p})H_{\eps_{0},i}$$
for an effective $\bQ$-divisor $H_{\eps_{0},i}$ on $F_i$ where $\theta_i:F_i\rw F_{i,0}$ is the contraction onto its minimal model for $i=1,2$.

Now since 
$$L_{4,i}-(4-\frac{2}{p})H_{\eps_{0},i}\equiv 3\theta_i^*(K_{F_{i,0}})+\delta_0\theta_i^*(K_{F_{i,0}}),$$
Lemma \ref{1} implies that 
$$|K_{F_i}+\roundup{L_{4,i}-(4-\frac{2}{p})H_{\eps_{0},i}}|$$
gives a birational map and so does $|K_{F_i}+\roundup{L_{4,i}}|$ for $i=1,2$. 

Since $F_1$ and $F_2$ are general fibers of $h$, $|5K_Y|$ gives a birational map. So we conclude the statement.\end{proof}

Now we are prepared to prove Theorem \ref{m1}. The outline of the proof follows that of Todorov \cite[Section 4]{Joro} except that the last step (i.e. {\bf Subcase 2.2}) is implemented by an alternative argument.  

\begin{proof}[{\bf Proof of Theorem \ref{m1}}]
Assume $\vol(V)>\alpha_0^3$ for some rational number $\alpha_0>0$. For any very small rational number $\eps>0$, take a birational modification $\mu:V'\rw V$ such that
there is a decomposition 
$$\mu^*(K_V)\sim_{\bQ} A+E_{\eps}$$ for some ample $\bQ$-divisor $A$ and some effective $\bQ$-divisor $E_{\eps}$  
and that the decomposition satisfies the properties of Takayama \cite[Theorem 3.1]{Tak}. (In particular, $\vol(A)\approx \vol(V)$, but $\vol(A)<\vol(V)$.) 

Take two different points $x_1, x_2\in V'$ at very general positions. Then there is a divisor $D\sim_{\bQ}\lambda A$ with $\lambda<\frac{3\sqrt[3]{2}}{(1-\eps)\alpha_0}$ such that $x_i\in \text{Nklt}(V',D)$ for $i=1,2$. 

Suppose the volume of those surfaces passing through very general points of $V'$ has the lower bound $\alpha_1^2$. Of course, $\alpha_1^2\geq 1$ and the volume of curves passing through very general points of $V'$ has the lower bound $\alpha_2\geq 2$. \medskip

{\it Case 1.} Assume $\alpha_1\geq 10$. According to Takayama \cite[Proposition 5.3]{Tak}, there is a constant $a_3<s_3+\frac{t_3}{\alpha_0}$  and an effective $\bQ$-divisor $D_3\sim_{\bQ} a_3A$ such that $x_1,x_2\in \text{Nklt}(V',D_3)$ and that at least one point $x_i$ is isolated in $\text{Nklt}(V',D_3)$. Thus Nadel vanishing implies that $|(\rounddown{a_3}+2)K_{V'}|$ gives a birational map. Unfortunately it is not enough to get the birationality of $|5K_V|$ whenever $\alpha_1$ is smaller. In fact, according to Takayama, one may take{\small
\begin{eqnarray*}
&&s_3=\frac{2}{1-\eps}+(1+\frac{1}{1-\eps})\cdot \frac{4\sqrt{2}}{(1-\eps)\alpha_1}+2\eps(1+\frac{1}{1-\eps})(1+\frac{\sqrt{2}}{(1-\eps)\alpha_1}),\\
&&t_3=(1+\frac{1}{1-\eps})(1+\frac{2\sqrt{2}}{(1-\eps)\alpha_1})\cdot \frac{3\sqrt[3]{2}}{1-\eps}.
\end{eqnarray*}
}By taking a very small $\eps$, while under the premise of $\alpha_0>12$ and $\alpha_1\geq 10$, one can easily verify $a_3<4$, which means $\varphi_5$ is birational.
\medskip

{\it Case 2.} Contrary to Case 1, we may assume $\text{Nklt}(V',D)$ contains an irreducible component which, passing through a very general point of $V'$, is a surface of volume $<10^2$, i.e. $\alpha_1<10$. As already realized by McKernan \cite{JM} and in Todorov \cite[p1328]{Joro}, there are two dominant morphisms $\nu:V''\rw V'$ and $f:V''\rw B$ where $V''$ is a normal projective 3-fold and $B$ is a normal projective curve. For a very general point $x\in V'$, there is a surface $V_x$ which is a pure centre of a {\it lc} pair $(V',D_x)$ with $D_x\sim_{\bQ} \lambda' A$ and $\lambda'<\frac{3\sqrt[3]{2}}{(1-\eps)\alpha_0}$ and, furthermore, $V_x$ is dominated by at least one general fiber of $f$. 
\medskip

\noindent{\bf Subcase 2.1.} Suppose $\nu$ is non-birational, which means that passing through a very general point $x\in V'$
there are at least two lc centers of the pairs $(V',D)$ and $(V', \tilde{D})$. Pick two different very general points $x_1(=x), x_2\in V'$. 
Using Todorov \cite[Lemma 3.3]{Joro}, one may find a divisor $D''\sim_{\bQ} \lambda''A$ with $\lambda''<\frac{9\sqrt[3]{2}}{(1-\eps)\alpha_0}+\eps$ such that 
$\text{Nklt}(V',D'')$ has at worst a 1-dimensional centre at $x_1=x$. This situation also fits into Takayama's induction \cite[Proposition 5.3]{Tak}. In fact, we may set 
$s_2=0$ and $t_2=\frac{9\sqrt[3]{2}}{1-\eps}+\eps\alpha_0$ in Takayama's Notation \cite[5.2(3)]{Tak} and begin the induction onwards. Then one gets a divisor $\hat{D}_3\sim_{\bQ}\hat{a}_3A$ with $\hat{a}_3<s_3+\frac{t_3}{\alpha_0}$ where{\small
\begin{eqnarray*}
&&s_3=(1+\frac{1}{1-\eps})\eps+\frac{2}{1-\eps},\\
&&t_3=(1+\frac{1}{1-\eps})(\eps\alpha_0+\frac{9\sqrt[3]{2}}{1-\eps})
\end{eqnarray*}}
such that $\text{codim}\text{Nklt}(V',\hat{D}_3)=3$ at $x_1$. Clearly one has{\small 
$$\hat{a}_3<\eps(1+\frac{1}{1-\eps})+\frac{2}{1-\eps}+
(1+\frac{1}{1-\eps})\cdot \frac{9\sqrt[3]{2}}{(1-\eps)\alpha_0}.$$
}Taking a very small $\eps>0$, one easily sees $\hat{a}_3<4$ as long as $\alpha_0>12$. Thus, by Nadel vanishing again, $\varphi_5$ is birational whenever $\vol(V)>12^3$ in this subcase.
\medskip

\noindent{\bf Subcase 2.2.} Suppose $\nu$ is birational. Without loss of generality, $V''$ can be considered smooth. We study the birationality of $|5K_{V''}|$ instead. We have a fibration $f:V''\rw B$ with $B$ a normal projective curve. By assumption a general fiber $S$ of $f$ has the property $\vol(S)<10^2.$ 

Since $\vol(V)>12^3$, we have 
$$\tau_0:=\frac{\vol(V)}{3\vol(S)}>\frac{12^3}{3\cdot 10^2}=\frac{144}{25}$$
and \cite[Lemma 2.7(1)]{Bonn} implies that
$$K_{V''}\geq (\tau_0-\eta)S$$
for any very small rational number $\eta>0$ (note here that the proof of \cite[Lemma 2.7(1)]{Bonn} works for any curve $B$). Take such a number $\eta$ so that $\tau_0-\eta>5$. Then we are in the situation of Proposition \ref{3} which says $\varphi_5$ is birational. We are done.
\end{proof}

Next we shall study the high dimensional analog.

\begin{prop}\label{4} Let $Y$ be a nonsingular projective $n$-fold ($n\geq 4$) of general type, $B$ a smooth projective curve and $f: Y\rw B$ a fibration. Denote by $X$ a general fiber of $f$. Assume $K_Y\sim_{\bQ} p_nX+E_{p_n}$ for some rational number $p_n>2\tilde{r}_{n-1}-2$ and for some effective $\bQ$-divisor $E_{p_n}$. 
Then $|mK_Y|$ gives a birational map for all $m\geq\tilde{r}_{n-1}$.
\end{prop}
\begin{proof} Take a minimal model $Y_0$ of $Y$. Modulo birational modifications over $Y$, we may assume for simplicity that there is a birational morphism $\pi:Y\rw Y_0$. Take a sufficiently large and divisible integer $m>0$ such that
\begin{itemize}
\item[1.] $mK_{Y_0}$ is a Cartier divisor, $mE_{p_n}$ is an integral divisor and $|mK_{Y_0}|$ is base point free;
\item[2.] $|p_nmX|$ is base point free;
\item[3.] $|p_nmK_{X_0}|$ is base point free where $X_0$ is a minimal model of $X$;
\item[4.] $mK_Y\sim_{\mathbb Z} mp_nX+mE_{p_n}$;
\item[5.] the support of the union of $E_{p_n}$ and all those exceptional divisors of $\pi$ is simple normal crossing.
\end{itemize}
Then one has $p_nmX\leq m\pi^*(K_{Y_0})\leq mK_Y$. 
Write $\pi^*(K_{Y_0})\sim_{\bQ}p_nX+E_{p_n}'$ where $E_{p_n}'$ is another effective $\bQ$-divisor. Pick general fibers $X$ and $X_1\neq X_2$ of $f$.
\medskip

\noindent{\it Step 1}. The canonical restriction inequality. 

If $g(B)>0$, one may take further modification $Y'$ of $Y$ (still assume $Y'=Y$) such that, for the fiber $X'$ over $X$ (still assume $X'=X$), there is a birational morphism $\nu: X\rw X_0$. Now the proof of Chen--Zuo \cite[Lemma 3.4]{CAG} implies
$$\pi^*(K_{Y_0})|_X\sim_{\bQ} \nu^*(K_{X_0}).$$
(Note that the key fact used in the proof of \cite[Lemma 3.4]{CAG} is the rational chain connectedness of Shokurov in dimension 3 and that was proved by Hacon--McKernan \cite[Corollary 1.3]{HM} for all dimensions.) 

If $g(B)=0$, we have the inclusion $\OO_B(p_nm)\hrw f_*\omega_Y^m$. Take a very large and divisible integer $l$ such that both $lK_{Y_0}$ and $lK_{X_0}$ are Cartier divisors. Then we have the inclusion
$$f_*\omega_{Y/B}^{lp_nm}\hrw f_*\omega_Y^{lp_nm+2lm}.$$
The semi-positivity theorem (see, for instance, Viehweg \cite{VS}) implies that $f_*\omega_{Y/B}^{lpm}$ is generated by global sections. Since $|\nu^*(lp_nmK_{X_0})|$ (as the moving part of $|lp_nmK_X|$) is base point free and the moving part of $|(lp_nm+2lm)K_Y|$ is exactly $|\pi^*((lp_nm+2lm)K_{Y_0})|$ thanks to the MMP, we have
$$\pi^*((lp_nm+2lm)K_{Y_0})|_X\geq \nu^*(lp_nmK_{X_0}).$$

In both cases, we have the following equality
\begin{equation}\label{e}
\pi^*(K_{Y_0})|_X\equiv \frac{p_n}{p_n+2}\nu^*(K_{X_0})+G_n
\end{equation}
for certain effective $\bQ$-divisor $G_n$ on $X$. 
\medskip

\noindent{\it Step 2}. The induction step by means of vanishing.  

Both $X_1$ and $X_2$ are known to be nonsingular projective $(n-1)$-folds of general type. Since 
$$(\tilde{r}_{n-1}-1)\pi^*(K_{Y_0})-X_1-X_2-\frac{2}{p_n}E_{p_n}'\equiv (\tdr_{n-1}-1-\frac{2}{p_n})\pi^*(K_{Y_0})$$ is nef and big by assumption, Kawamata--Viehweg vanishing gives the surjective map{\small 
\begin{equation}\label{4.1}
\begin{array}{rl}
&H^0(K_Y+\roundup{K_{n,p_n}})\\
\lrw&H^0(X_1, K_{X_1}+\roundup{K_{n,p_n}-X_1}|_{X_1})
\oplus H^0(X_2, K_{X_2}+\roundup{K_{n,p_n}-X_2}|_{X_2})
\end{array}\end{equation}
}where $K_{n,p_n}:=(\tdr_{n-1}-1)\pi^*(K_{Y_0})-\frac{2}{p_n}E_{p_n}'$.
Observing that
$$\roundup{(\tdr_{n-1}-1)\pi^*(K_{Y_0})-\frac{2}{p_n}E_{p_n}'-X_i}|_{X_i}\geq 
\roundup{L_{n,i}}$$
where{\small $$L_{n,i}:=((\tdr_{n-1}-1)\pi^*(K_{Y_0})-\frac{2}{p_n}E_{p_n}'-X_i)|_{X_i}\equiv (\tdr_{n-1}-1-\frac{2}{p_n})\pi^*(K_{Y_0})|_{X_i}$$ }for $i=1,2$,  we have{\small 
$$|\tdr_{n-1}K_Y||_{X_i}\succeq |K_Y+\roundup{(\tdr_{n-1}-1)\pi^*(K_{Y_0})-\frac{2}{p_n}E_{p_n}'}||_{X_i}\succeq |K_{X_i}+\roundup{L_{n,i}}|.$$
}Due to the surjectivity of (\ref{4.1}), it is sufficient to show that $|K_{X_i}+\roundup{L_{n,i}}|$ gives a birational map for $i=1,2$. 
\medskip

\noindent{\it Step 3}. Verification on the general fiber.

Since $p_n>2\tdr_{n-1}-2$, we have $\frac{p_n}{p_n+2}>\frac{\tdr_{n-1}-1}{\tdr_{n-1}}$ and $\delta_0:=(\tdr_{n-1}-1-\frac{2}{p_n})\frac{p_n}{p_n+2}-(\tdr_{n-1}-2)>0$. 
Now we have
$$L_{n,i}\equiv (\tdr_{n-1}-2+\delta_0)\theta_i^*(K_{X_{i,0}})+(\tdr_{n-1}-1-\frac{2}{p_n})G_{n,i}$$
for an effective $\bQ$-divisor $G_{n,i}$ on $X_i$ where $\theta_i:X_i\rw X_{i,0}$ is the contraction onto its minimal model $X_{i,0}$ for $i=1,2$.
Since 
$$L_{n,i}-(\tdr_{n-1}-1-\frac{2}{p_n})G_{n,i}\equiv (\tdr_{n-1}-2)\theta_i^*(K_{X_{i,0}})+\delta_0\theta_i^*(K_{X_{i,0}}),$$
one knows that
$$|K_{X_i}+\roundup{L_{n,i}-(\tdr_{n-1}-1-
\frac{2}{p_n})G_{n,i}}|$$
gives a birational map by the definition of $\tdr_{n-1}$ and so does $|K_{X_i}+\roundup{L_{n,i}}|$ for $i=1,2$. 

Since $X_1$ and $X_2$ are general fibers of $f$, $|\tdr_{n-1}K_Y|$ gives a birational map. So we are done.
\end{proof}

\begin{lem}\label{3+} For all $n\geq 3$, $\tdr_n<+\infty$.
\end{lem}
\begin{proof} We have already seen $\tdr_2=5$ in Lemma \ref{1}. For $n>2$, we prove this by an induction on the dimensions. 
We sketch the proof here.

By the boundedness theorem of Hacon--McKernan, Takayama and Tsuji, we can take a positive integer $m_0$ so that $P_{m_0}(X_0)\geq 2$ for 
all minimal projective n-folds $X_0$ of general type.  Take a pencil $\Lambda\subset |m_0K_{X_0}|$. Given any birational morphism $\nu:X'\rw X_0$ and given any nef and big $\bQ$-divisor $R$ on $X'$, 
take a birational modification $\nu_1:X_1\rw X_0$ such that:
\begin{itemize}
\item[(1)] $X_1$ is smooth;
\item[(2)] $\nu_1$ factors through $\nu$, i.e. $\nu_1=\nu\circ \nu'$ and thus $X_1$ is over $X'$;
\item[(3)] the moving part of $\nu_1^*(\Lambda)$ is base point free;
\item[(4)] the support of union of exceptional divisors of $\nu_1$ and the pull back of $\{R\}$ is simple normal crossing.
\end{itemize}
Now it is sufficient to find some positive number $m$ (independent of $X_0$, the choice of $\nu$ and the $\bQ$-divisor $R$) so that
$$|K_{X_1}+\roundup{(m-2)\nu_1^*(K_{X_0})+\nu'^*(R)}|$$
gives a birational map. 

Take two generic irreducible elements $F_1$ and $F_2$ in the moving part of  $\nu_1^*(\Lambda)$. In other words, the $F_i$ is a general fiber of the induced fibration after taking the Stein factorization of $\Phi_{\Lambda}\circ \nu_1$.  By definition, $F_1$ and $F_2$ are irreducible elements in a free pencil. 
Both $F_1$ and $F_2$ are smooth by Bertini's theorem and they are of general type with dimensions $n-1$. For any integer $m\geq 2m_0+3$, since $(m-2)\nu_1^*(K_{X_0})+\nu'^*R\geq (m-2m_0-2)\nu_1^*(K_{X_0})+\nu'^*R+F_1+F_2$, Kawamata--Viehweg vanishing gives the surjective map:
\begin{equation}\label{sur}
\begin{array}{lll}
&&H^0(X_1, K_{X_1}+\roundup{(m-2m_0-2)\nu_1^*(K_{X_0})+\nu'^*R}+F_1+F_2)\\
&\lrw& \oplus_{i=1}^2 H^0(F_i, K_{F_i}+\roundup{(m-2m_0-2)\nu_1^*(K_{X_0})+\nu'^*R}|_{F_i})
\end{array}\end{equation}

Set $F=F_i$ and we study the restricted linear system on $F$. Modulo further birational modifications, we may assume that there is a birational contraction morphism $\sigma: F\rw F_0$ onto a minimal model $F_0$. By the canonical restriction inequality in Chen \cite[Lemma 2.5 and Inequality (1) at page 225]{KIAS}, we have
$$\nu_1^*(K_{X_0})|_F\geq \frac{1}{2m_0+1}\sigma^*(K_{F_0}).$$
Thus 
$$K_{F}+\roundup{q\nu_1^*(K_{X_0})+\nu'^*R}|_F\geq
K_F+\roundup{t\sigma^*(K_{F_0})+Q}$$
where $q:=m-2m_0-2$, $t:=\rounddown{\frac{q}{2m_0+1}}$ is an integer depending only on $m$ and $Q$ is a certain nef and big $\bQ$-divisor on $F$. By the induction, the statement is true on $F$ as long as $t$ is large enough and such an integer $t$ should work for all $(n-1)$-fold $F$.  The surjective map \ref{sur} says that, for some integer $m$,  the linear system 
$|K_{X_1}+\roundup{(m-2)\nu_1^*(K_{X_0})+\nu'^*(R)}|$ separates different general $F_1$ and $F_2$, and it separates general points of $F_i$ for $i=1,2$. 
We are done.
\end{proof}

With Proposition \ref{4} and Lemma \ref{3+}, we are prepared for proving Theorem \ref{m2}.  

%\begin{them}\label{m2} Let $Y$ be a nonsingular projective 4-fold of general type. Assume $\vol(Y)\gg 0$, say $\vol(Y)>\frac{\tdr_3-1}{6}\cdot 12^4$. Then $|mK_Y|$ gives a birational map for all $m\geq \tdr_3$.
%\end{them}
\begin{proof}[{\bf Proof of Theorem \ref{m2}}] We sketch the proof while omitting redundant calculations.  We keep a similar procedure to the proof of Theorem \ref{m1}. 

{}First of all we know $\tdr_3\geq 27$ due to Iano-Fletcher's example $X_{46}\subset \bP(4,5,6,7,23)$ in \cite[p.151]{Fletcher}. 

Then one uses Takayama's effective induction in \cite[Proposition 5.3]{Tak} and McKernan's dominant morphism $\nu:V''\rw V'$ in \cite[Lemma 3.2]{JM} to exclusively treat 3 cases with regard to $\text{Nklt}(V',D)$, where $V'$ is certain birational model of $Y$ as in the proof of Theorem \ref{m1} replacing $V$ by $Y$:
\begin{itemize}
\item[i.] Only sub-varieties of dimension $\leq 2$ passing through very general points of $Y$ need to be considered. Note that the situation corresponding to Subcase 2.1 in the proof of Theorem \ref{m1} is included here. Meanwhile, $\nu$ is non-birational by assumption. We may use Todorov \cite[Lemma 3.3]{Joro} to find a divisor $D''\sim \lambda''A$ such that $\text{Nklt}(V',D'')$ has at worst a 1-dimensional lc center passing through $x$--a very general point of $V'$. Replace $(V',D)$ by $(V', D'')$ and go on the procedure of applying Takayama's induction.
\item[ii.] Apart from i, there is a 3-fold center, with a large volume,  of the pair $(V',D_x)$ passing through a very general point $x\in V'$.
\item[iii.] Due to McKernan \cite{JM}, $Y$ is birationally fibred by 3-folds of small volumes over a curve $B$.
\end{itemize}

For Case i and Case ii, it is easy to find the constraint on $\vol(Y)$ (say $\vol(Y)\geq \alpha_0^4>0$) so that $|27K_Y|$ gives a birational map by assuming that all those 3-folds passing through very general points of $Y$ have volume $>(3\sqrt[3]{16})^3$. In fact, the worse situation in Case i is when $\nu$ is non-birational for which we use Todorov \cite[Lemma 3.3]{Joro}. So we may set $s_2=0$, $t_2=12\sqrt[4]{2}$, while $\varepsilon\mapsto 0$, and then Takayama's induction gives, $s_4<8\sqrt{2}+2$ and 
$$t_4<24\sqrt[4]{2}(2\sqrt{2}+1).$$ 
Thus we have $a_4<s_4+\frac{t_4}{\alpha_0}$ and $\roundup{a_4}+1<27$ as long as $\alpha_0$ is large enough. Actually this is the case when $\alpha_0^4>\frac{\tdr_3-1}{6}\cdot 12^4\geq \frac{13}{3}\cdot 12^4$. Thus $|27K_Y|$ gives a birational map. In Case ii, by assuming $\alpha_1>3\sqrt[3]{4}$, Takayama's induction gives:
$$s_4<8\sqrt{2}+12(1+2\sqrt{2})\sqrt[3]{2}/\alpha_1,$$
$$t_4<8\sqrt[4]{2}(2\sqrt{2}+1)(1+3\sqrt[3]{2}/\alpha_1)$$
while letting $\varepsilon\mapsto 0$. {}Finally we have $a_4<s_4+t_4/\alpha_0$. Still under the assumption  
$\alpha_0^4>\frac{13}{3}\cdot 12^4$, it is easy to see $\roundup{a_4}+1<27$. Thus we see the birationality of $\Phi_{|27K_Y|}$. 

For Case iii, we may use Proposition \ref{4} under the condition that $Y$ is birationally fibred by 3-folds $X$ of small volume (say $\vol(X)\leq 16\times 27$). Now using a similar method to that of \cite[Lemma 2.7(1)]{Bonn} by computing the ratio $\frac{\vol(Y)}{4\vol(X)}$, it is easy for us to find the condition on $\vol{(Y)}$ again. In fact, $\vol(Y)>\frac{\tdr_3-1}{6}\cdot 12^4$ will do for this case to get the birationality of $|mK_Y|$ for all $m\geq \tdr_3$. We are done.
\end{proof}

Thanks to the elaborate calculation in Di Biagio \cite[Theorem 5.9]{DB}, we can save several pages to present here our effective result for all dimensions.  So far we have the following direct result in higher dimensions:

\begin{cor}(=Corollary \ref{m3}) There are computable constants $C_n$ for all $n\geq 5$. Let $Y$ be a nonsingular projective n-fold with $\vol(Y)>C_n$. Then $|mK_Y|$ gives a birational map onto its image for all $m\geq R_n$. 
\end{cor}
We omit the proof since it can be obtained by a similar argument to that of Theorem \ref{m1} and Theorem \ref{m2}.
This may serve as an interesting exercise.
\medskip

{}Finally we note that $R_3=\tdr_2=r_2=5$ by Bombieri and, in fact, $R_4=\tdr_3\geq r_3$. It is natural to ask the following:
\medskip

\noindent{\bf Open Problem $X_n$.} {\em Is it always true that $R_n=\tdr_{n-1}=r_{n-1}$ for all $n\geq 4$?}
\medskip

A positive answer to Problem $X_n$ should be expectable.

\section*{\bf Acknowledgments}

I appreciate the generous support of Max-Planck-Institut f$\ddot{\text{u}}$r Mathematik (Bonn) during my four months visit there in 2011--2012. I would like to thank Christopher D. Hacon for valuable comments to the first draft of this note.  The project was supported by National Natural Science Foundation of China (\#11171068),  Doctoral Fund of Ministry of Education of China (\#20110071110003) and partially by NSFC for Innovative Research Groups (\#11121101).

\end{document}